\documentclass[11pt]{amsart}
\makeatletter
\def\subsection{\@startsection{subsection}{3}
  \z@{.5\linespacing\@plus.7\linespacing}{.1\linespacing}
  {\normalfont\itshape}}
\makeatother

\usepackage{amssymb}
\usepackage{amsmath}
\usepackage{mathrsfs}
\usepackage{amsthm}
\usepackage{enumerate}
\newtheorem{theorem}{Theorem}[section]
\newtheorem{lemma}[theorem]{Lemma}

\newtheorem{proposition}[theorem]{Proposition}
\theoremstyle{definition}

\newtheorem*{remark}{Remark}
\newtheorem*{remarks}{Remarks}
\numberwithin{equation}{section}
\mathchardef\hyphen="2D
\def\@tvsp{\mathchoice{{}\mkern-4.5mu}{{}\mkern-4.5mu}{{}\mkern-2.5mu}{}}
\def\ln{\left|\@tvsp\left|\@tvsp\left|}
\def\rn{\right|\@tvsp\right|\@tvsp\right|}
\makeatother
\allowdisplaybreaks
\begin{document}
\title{Sample covariance matrices converge to compound free Poisson distribution}
\author{March T.~Boedihardjo}
\address{Department of Mathematics, Texas A\&M University, College Station, Texas 77843}
\email{march@math.tamu.edu}
\subjclass[2010]{}
\begin{abstract}
We show that the empirical distribution of the eigenvalues of the sample covariance matrix of certain random vectors (not necessarily independent entries) with bounded
marginal $L^{4}$ norms converges weakly to a compound free Poisson distribution.
\end{abstract}
\maketitle
\section{Main result}
Marchenko and Pastur \cite{Marchenko} showed that the empirical distribution of the eigenvalues of the sample covariance matrix of a random vector uniformly distributed on
the unit sphere converges weakly to the Marchenko-Pastur law. There has been many generalizations to general random vectors (see \cite{Adamczak}). The main result of this
paper is
\begin{theorem}\label{11}
Suppose that $f_{1},\ldots,f_{N}$ are independent random vectors on $\mathbb{C}^{n}$ such that
\[\sup_{x\in S^{n-1}}\mathbb{E}|(f_{j},x)|^{4}\leq\frac{L}{n^{2}}\text{ and }\mathbb{E}\|f_{j}\|^{k}\leq L_{k},\quad j=1,\ldots,N,\;k\geq 1\]
for some $L>0$ and $L_{k}>0$, $k\geq 1$ independent of $n$ and $N$. If $n,N\to\infty$ in such a way that $\frac{n}{N}\to\lambda\in(0,\infty)$ and
\[\left\|\sum_{j=1}^{N}\mathbb{E}\|f_{j}\|^{2(k-1)}f_{j}\otimes f_{j}-a_{k}I\right\|\leq Cn^{-\epsilon_{0}},\quad k\geq 1,\]
for some $a_{k}\in\mathbb{C}$, $k\geq 1$ and $C,\epsilon_{0}>0$ independent of $n$ and $N$, then
\[\mathbb{E}\circ\mathrm{tr}(f_{1}\otimes f_{1}+\ldots+f_{N}\otimes f_{N})^{p}\to\sum_{\pi\in\mathrm{NC}(p)}\prod_{B\in\pi}a_{|B|}.\]
\end{theorem}
Notation: tr means normalized trace. $\mathrm{NC}(p)$ is the set of all noncrossing partitions on $\{1,\ldots,p\}$.
\begin{remarks}
1. An immediate consequence of Theorem \ref{11} is that the theorem of Marchenko and Pastur still holds if the random vector is distributed (but not uniformly distributed)
on the unit sphere provided that it has bounded marginal $L^{4}$ norms.\\[\baselineskip]
2. The condition $\displaystyle\sup_{x\in S^{n-1}}\mathbb{E}|(f_{i},x)|^{4}\leq\frac{L}{n^{2}}$ cannot be removed from Theorem \ref{11}. For example, when $N=n$ and each
$f_{i}$ is uniformly distributed on the canonical basis $\{e_{i}\}_{i=1}^{n}$ for $\mathbb{C}^{n}$, we have $a_{k}=1$ and
\[\mathbb{E}\circ\mathrm{tr}(f_{1}\otimes f_{1}+\ldots+f_{n}\otimes f_{n})^{p}\to B_{p},\]
where $B_{p}$ is the Bell number, the number of partitions on $\{1,\ldots,p\}$.
\end{remarks}
\section{A graph inequality}
This section is devoted to proving the following lemma.
\begin{lemma}\label{21}
Let $S_{1},\ldots,S_{r}$ be subsets of a set $E$ such that every element $e\in E$ is contained in exactly two of the sets $S_{1},\ldots,S_{r}$. Assume that $|S_{1}|\leq
\ldots\leq|S_{r}|$. Let $t\geq 0$. Then
\[\min(t,|S_{1}|)+\min(t,|S_{2}\backslash S_{1}|)+\ldots+\min(t,|S_{r}\backslash(S_{1}\cup\ldots\cup S_{r-1})|)\geq\frac{\min(t,|S_{1}|)}{2}r.\]
\end{lemma}
\begin{lemma}\label{22}
Let $S_{1},\ldots,S_{r}$ be subsets of a set $E$ such that every element $x\in E$ is contained in exactly two of the sets $S_{1},\ldots,S_{r}$. Then
\[|E|=\frac{1}{2}\sum_{k=1}^{r}|S_{k}|.\]
\end{lemma}
\begin{proof}
By assumption, $\displaystyle\sum_{k=1}^{r}I_{S_{k}}(x)=2$ for all $x\in E$. So
\[\sum_{k=1}^{r}|S_{k}|=\sum_{k=1}^{r}\sum_{x\in E}I_{S_{k}}(x)=\sum_{x\in E}\sum_{k=1}^{r}I_{S_{k}}(x)=\sum_{x\in E}2=2|E|.\]
\end{proof}
In Lemma \ref{23} and \ref{25} below, $\Lambda^{c}$ is understood as $\{1,\ldots,r\}\backslash\Lambda$. Also when $k=1$, $S_{k}\backslash(S_{1}\cup\ldots\cup S_{k-1})$ is
understood as $S_{1}$.
\begin{lemma}\label{23}
Let $S_{1},\ldots,S_{r}$ be subsets of a set $E$ such that every element $x\in E$ is contained in exactly two of the sets $S_{1},\ldots,S_{r}$. If $\Lambda\subset
\{1,\ldots,r\}$ and $1\leq k_{0}\leq r$, then
\[\sum_{k\in\Lambda}|S_{k}\backslash(S_{1}\cup\ldots\cup S_{k-1})|\geq\frac{1}{2}\left(\sum_{\substack{1\leq k\leq k_{0}-1\\k\in\Lambda}}|S_{k}|-\sum_{\substack{1\leq k
\leq k_{0}-1\\k\in\Lambda^{c}}}|S_{k}|\right).\]
\end{lemma}
\begin{proof}
\begin{align*}
&\sum_{k\in\Lambda}|S_{k}\backslash(S_{1}\cup\ldots\cup S_{k-1})|\\=&\sum_{k=1}^{r}|S_{k}\backslash(S_{1}\cup\ldots\cup S_{k-1})|-\sum_{k\in\Lambda^{c}}|S_{k}\backslash
(S_{1}\cup\ldots\cup S_{k-1})|\\=&|E|-\sum_{k\in\Lambda^{c}}|S_{k}\backslash(S_{1}\cup\ldots\cup S_{k-1})|\text{ since }E=\bigcup_{k=1}^{r}S_{k}\\=&
\frac{1}{2}\sum_{k=1}^{r}|S_{k}|-\sum_{k\in\Lambda^{c}}|S_{k}\backslash(S_{1}\cup\ldots\cup S_{k-1})|\text{ by Lemma }\ref{22}\\=&
\frac{1}{2}\sum_{k\in\Lambda}|S_{k}|+\frac{1}{2}\sum_{k\in\Lambda^{c}}|S_{k}|-\frac{1}{2}\sum_{k\in\Lambda^{c}}|S_{k}\backslash(S_{1}\cup\ldots\cup S_{k-1})|-
\frac{1}{2}\sum_{k\in\Lambda^{c}}|S_{k}\backslash(S_{1}\cup\ldots\cup S_{k-1})|\\=&
\frac{1}{2}\sum_{k\in\Lambda}|S_{k}|+\frac{1}{2}\sum_{k\in\Lambda^{c}}|S_{k}\cap(S_{1}\cup\ldots\cup S_{k-1})|-\frac{1}{2}\sum_{k\in\Lambda^{c}}|S_{k}\backslash(S_{1}
\cup\ldots S_{k-1})|\\=&
\frac{1}{2}\sum_{k\in\Lambda}|S_{k}|+\frac{1}{2}\sum_{k\in\Lambda^{c}}|S_{k}\cap(S_{1}\cup\ldots\cup S_{k-1})|-\frac{1}{2}\sum_{\substack{1\leq k\leq k_{0}-1\\ k\in\Lambda^{c
}}}|S_{k}\backslash(S_{1}\cup\ldots S_{k-1})|\\&-\frac{1}{2}\sum_{\substack{k_{0}\leq k\leq n\\k\in\Lambda^{c}}}|S_{k}\backslash(S_{1}\cup\ldots\cup S_{k-1})|\\\geq&
\frac{1}{2}\sum_{k\in\Lambda}|S_{k}|+\frac{1}{2}\sum_{k\in\Lambda^{c}}|S_{k}\cap(S_{1}\cup\ldots\cup S_{k-1})|-\frac{1}{2}\sum_{\substack{1\leq k\leq k_{0}-1\\k\in\Lambda^{c}
}}|S_{k}|\\&-\frac{1}{2}\sum_{k_{0}\leq k\leq n}|S_{k}\backslash(S_{1}\cup\ldots\cup S_{k-1})|\\=&
\frac{1}{2}\sum_{\substack{1\leq k\leq k_{0}-1\\k\in\Lambda}}|S_{k}|+\frac{1}{2}\sum_{\substack{k_{0}\leq k\leq n\\k\in\Lambda}}|S_{k}|+\frac{1}{2}\sum_{k\in\Lambda^{c}}
|S_{k}\cap(S_{1}\cup\ldots\cup S_{k-1})|-\frac{1}{2}\sum_{\substack{1\leq k\leq k_{0}-1\\k\in\Lambda^{c}}}|S_{k}|
\\&-\frac{1}{2}\sum_{k_{0}\leq k\leq n}|S_{k}\backslash(S_{1}\cup\ldots\cup S_{k-1})|\\=&
\frac{1}{2}\left(\sum_{\substack{1\leq k\leq k_{0}-1\\k\in\Lambda}}|S_{k}|-\sum_{\substack{1\leq k\leq k_{0}-1\\k\in\Lambda^{c}}}|S_{k}|\right)+\\&
\frac{1}{2}\left(\sum_{\substack{k_{0}\leq k\leq n\\k\in\Lambda}}|S_{k}|+\sum_{k\in\Lambda^{c}}|S_{k}\cap(S_{1}\cup\ldots\cup S_{k-1})|-
\sum_{k_{0}\leq k\leq n}|S_{k}\backslash(S_{1}\cup\ldots\cup S_{k-1})|\right).
\end{align*}
To complete the proof, it suffices to show that
\begin{equation}\label{21e}
\sum_{\substack{k_{0}\leq k\leq r\\k\in\Lambda}}|S_{k}|+\sum_{k\in\Lambda^{c}}|S_{k}\cap(S_{1}\cup\ldots\cup S_{k-1})|-
\sum_{k_{0}\leq k\leq r}|S_{k}\backslash(S_{1}\cup\ldots\cup S_{k-1})|\geq 0.
\end{equation}
To begin,
\begin{align}\label{22e}
&\sum_{\substack{k_{0}\leq k\leq r\\k\in\Lambda}}|S_{k}|+\sum_{k\in\Lambda^{c}}|S_{k}\cap(S_{1}\cup\ldots\cup S_{k-1})|-
\sum_{k_{0}\leq k\leq r}|S_{k}\backslash(S_{1}\cup\ldots\cup S_{k-1})|\nonumber\\\geq&
\sum_{\substack{k_{0}\leq k\leq r\\k\in\Lambda}}|S_{k}\cap(S_{1}\cup\ldots\cup S_{k-1})|+\sum_{\substack{k_{0}\leq k\leq r\\k\in\Lambda^{c}}}|S_{k}\cap(S_{1}\cup\ldots\cup
S_{k-1})|\nonumber\\&-\sum_{k_{0}\leq k\leq r}|S_{k}\backslash(S_{1}\cup\ldots\cup S_{k-1})|\nonumber\\=&
\sum_{k_{0}\leq k\leq r}|S_{k}\cap(S_{1}\cup\ldots\cup S_{k-1})|-\sum_{k_{0}\leq k\leq r}|S_{k}\backslash(S_{1}\cup\ldots\cup S_{k-1})|\nonumber\\=&
\sum_{k_{0}\leq j\leq r}|S_{j}\cap(S_{1}\cup\ldots\cup S_{j-1})|-\sum_{k_{0}\leq k\leq r}|S_{k}\backslash(S_{1}\cup\ldots\cup S_{k-1})|.
\end{align}
By assumption, every element in $V$ is contained in at least two of the sets $S_{1},\ldots,S_{r}$. Therefore, if an element $e$ of $S_{k}$ is not in $S_{1}\cup\ldots\cup S_{k-1}$
then $e$ must be in $S_{k+1}\cup\ldots\cup S_{r}$. Thus,
\[|S_{k}\backslash(S_{1}\cup\ldots\cup S_{k-1})|\leq|S_{k}\cap(S_{k+1}\cup\ldots\cup S_{r})|\leq\sum_{k+1\leq j\leq r}|S_{k}\cap S_{j}|.\]
Hence,
\begin{eqnarray}\label{23e}
\sum_{k_{0}\leq k\leq r}|S_{k}\backslash(S_{1}\cup\ldots\cup S_{k-1})|&\leq&\sum_{k_{0}\leq k\leq r}\sum_{k+1\leq j\leq n}|S_{k}\cap S_{j}|\nonumber\\&=&
\sum_{k_{0}+1\leq j\leq r}\sum_{k_{0}\leq k\leq j-1}|S_{k}\cap S_{j}|\nonumber\\&\leq&
\sum_{k_{0}\leq j\leq r}\sum_{1\leq k\leq j-1}|S_{k}\cap S_{j}|.
\end{eqnarray}
By assumption, every element in $E$ is contained in at most two of the sets $S_{1},\ldots,S_{n}$. So the sets $S_{1}\cap S_{j},\ldots,S_{j-1}\cap S_{j}$ are disjoint.
So $\displaystyle\sum_{1\leq k\leq j-1}|S_{k}\cap S_{j}|=|S_{j}\cap(S_{1}\cup\ldots\cup S_{j-1})|$. Thus, by (\ref{23e}),
\[\sum_{k_{0}\leq k\leq r}|S_{k}\backslash(S_{1}\cup\ldots\cup S_{k-1})|\leq\sum_{k_{0}\leq j\leq r}|S_{j}\cap(S_{1}\cup\ldots\cup S_{j-1})|.\]
Combining this with (\ref{22e}), we obtain (\ref{21e}). This completes the proof.
\end{proof}
\begin{lemma}\label{24}
Let $m\geq 1$. Let $\Lambda_{1}$ and $\Lambda_{2}$ be subsets of $\{1,\ldots,m\}$. If $|[l,m]\cap\Lambda_{1}|\leq|[l,m]\cap\Lambda_{2}|$ for all $l\in\{1,\ldots,m\}$ then
there exists a strictly increasing function $f:\Lambda_{1}\to\Lambda_{2}$ such that $f(k)\geq k$ for all $k\in\Lambda_{1}$.
\end{lemma}
\begin{proof}
Since by assumption $|\Lambda_{1}|\leq|\Lambda_{2}|$, the function $f:\Lambda_{1}\to\Lambda_{2}$ defined by sending the $i$th largest element of $\Lambda_{1}$ to the $i$th
largest element of $\Lambda_{2}$ is well defined and strictly increasing. It remains to show that $f(k)\geq k$ for all $k\in\Lambda_{1}$. For each $i=1,\ldots,
|\Lambda_{1}|$, let $k_{i}$ be the $i$th largest element of $\Lambda_{1}$. By assumption, $|[k_{i},m]\cap\Lambda_{1}|\leq|[k_{i},m]\cap\Lambda_{2}|$ for all $i=1,\ldots,
|\Lambda_{1}|$. Note that $[k_{i},m]\cap\Lambda_{1}=\{k_{1},k_{2}\ldots,k_{i}\}$. So $|[k_{i},m]\cap\Lambda_{1}|=i$. Therefore, $|[k_{i},m]\cap\Lambda_{2}|\geq i$ for all $i=1,
\ldots,|\Lambda_{1}|$. So the $i$th largest element of $\Lambda_{2}$ is at least $k_{i}$. So $f(k_{i})\geq k_{i}$ for all $i=1,\ldots,|\Lambda_{1}|$ so $f(k)\geq k$ for all
$k\in\Lambda_{1}$.
\end{proof}
\begin{lemma}\label{25}
Let $S_{1},\ldots,S_{r}$ be subsets of a set $E$ such that every element $x\in E$ is contained in exactly two of the sets $S_{1},\ldots,S_{r}$. Assume that $|S_{1}|\leq
\ldots\leq|S_{r}|$. If $\Lambda\subset\{1,\ldots,r\}$ then
\[\sum_{k\in\Lambda}|S_{k}\backslash(S_{1}\cup\ldots\cup S_{k-1})|\geq\frac{1}{2}|S_{1}|(|\Lambda|-|\Lambda^{c}|).\]
\end{lemma}
\begin{proof}
Case I: {\it For every $1\leq l\leq r$, $|[l,r]\cap\Lambda^{c}|<|[l,r]\cap\Lambda|$.}\\
From the first four lines of the proof of Lemma \ref{23}, we have
\[\sum_{k\in\Lambda}|S_{k}\backslash(S_{1}\cup\ldots\cup S_{k-1})|=\frac{1}{2}\sum_{k=1}^{r}|S_{k}|-\sum_{k\in\Lambda^{c}}|S_{k}\backslash(S_{1}\cup\ldots\cup S_{k-1})|.\]
Thus,
\[\sum_{k\in\Lambda}|S_{k}\backslash(S_{1}\cup\ldots\cup S_{k-1})|\geq\frac{1}{2}\sum_{k=1}^{r}|S_{k}|-\sum_{k\in\Lambda^{c}}|S_{k}|=\frac{1}{2}\sum_{k\in\Lambda}|S_{k}|-
\frac{1}{2}\sum_{k\in\Lambda^{c}}|S_{k}|.\]
Taking $m=r$, $\Lambda_{1}=\Lambda^{c}$ and $\Lambda_{2}=\Lambda$ in Lemma \ref{24}, we obtain an injective function $f:\Lambda^{c}\to\Lambda$ such that $f(k)\geq k$ for
all $k\in\Lambda^{c}$. Therefore,
\begin{eqnarray*}
\sum_{k\in\Lambda}|S_{k}\backslash(S_{1}\cup\ldots\cup S_{k-1})|&=&\frac{1}{2}\sum_{j\in\Lambda}|S_{j}|-\frac{1}{2}\sum_{k\in\Lambda^{c}}|S_{k}|\\&=&
\frac{1}{2}\sum_{j\in f(\Lambda^{c})}|S_{j}|+\frac{1}{2}\sum_{j\in\Lambda\backslash f(\Lambda^{c})}|S_{j}|-\frac{1}{2}\sum_{k\in\Lambda^{c}}|S_{k}|\\&=&
\frac{1}{2}\sum_{k\in \Lambda^{c}}|S_{f(k)}|+\frac{1}{2}\sum_{j\in\Lambda\backslash f(\Lambda^{c})}|S_{j}|-\frac{1}{2}\sum_{k\in\Lambda^{c}}|S_{k}|\\&=&
\frac{1}{2}\sum_{k\in \Lambda^{c}}(|S_{f(k)}|-|S_{k}|)+\frac{1}{2}\sum_{j\in\Lambda\backslash f(\Lambda^{c})}|S_{j}|\\&\geq&
0+\frac{1}{2}|\Lambda\backslash f(\Lambda^{c})||S_{1}|.
\end{eqnarray*}
The last inequality follows from the fact that $f(k)\geq k$ for all $k\in\Lambda^{c}$ and the assumption that $|S_{1}|\leq\ldots\leq|S_{r}|$. Since $|\Lambda\backslash
f(\Lambda^{c})|=|\Lambda|-|f(\Lambda^{c})|=|\Lambda|-|\Lambda^{c}|$, it follows that
\[\sum_{k\in\Lambda}|S_{k}\backslash(S_{1}\cup\ldots\cup S_{k-1})|\geq\frac{1}{2}(|\Lambda|-|\Lambda^{c}|)|S_{1}|.\]
Case II: {\it There exists $1\leq k_{0}\leq r$ such that $|[k_{0},r]\cap\Lambda^{c}|\geq |[k_{0},r]\cap\Lambda|$.}\\
We may assume that $k_{0}$ is the smallest one with such property. We may also assume that $k_{0}>1$. Otherwise, the result is trivial. Thus, we have
$|[l,k_{0}-1]\cap\Lambda^{c}|<|[l,k_{0}-1]\cap\Lambda|$ for all $l\in\{1,\ldots,k_{0}-1\}$. Otherwise, an $l$ failing this property would contradict with the minimality of
$k_{0}$. Taking $m=k_{0}-1$, $\Lambda_{1}=[1,k_{0}-1]\cap\Lambda^{c}$ and $\Lambda_{2}=[1,k_{0}-1]\cap\Lambda$ in Lemma \ref{24}, we obtain an injective function
$f:[1,k_{0}-1]\cap\Lambda^{c}\to[1,k_{0}-1]\cap\Lambda$ satisfying $f(k)\geq k$ for all $k\in[1,k_{0}-1]\cap\Lambda^{c}$.

By Lemma \ref{23}, we have
\begin{align*}
&\sum_{k\in\Lambda}|S_{k}\backslash(S_{1}\cup\ldots\cup S_{k-1})|\\\geq&\frac{1}{2}\left(\sum_{\substack{1\leq k\leq k_{0}-1\\k\in\Lambda}}|S_{k}|-\sum_{\substack{1\leq k
\leq k_{0}-1\\k\in\Lambda^{c}}}|S_{k}|\right)\\=&\frac{1}{2}\left(\sum_{j\in[1,k_{0}-1]\cap\Lambda}|S_{j}|-\sum_{k\in[1,k_{0}-1]\cap\Lambda^{c}}|S_{k}|\right)\\=&
\frac{1}{2}\left(\sum_{j\in\{f(k):k\in[1,k_{0}-1]\cap\Lambda^{c}\}}|S_{j}|+\sum_{j\in[1,k_{0}-1]\cap\Lambda\backslash\{f(k):k\in[1,k_{0}-1]\cap\Lambda^{c}\}}|S_{j}|-\sum_{k\in[1,k_{0}-1]\cap\Lambda^{c}}|S_{k}|\right)\\=&
\frac{1}{2}\left(\sum_{k\in[1,k_{0}-1]\cap\Lambda^{c}}|S_{f(k)}|+\sum_{j\in[1,k_{0}-1]\cap\Lambda\backslash\{f(k):k\in[1,k_{0}-1]\cap\Lambda^{c}\}}|S_{j}|-\sum_{k\in[1,k_{0}-1]\cap\Lambda^{c}}|S_{k}|\right)\\=&
\frac{1}{2}\left(\sum_{k\in[1,k_{0}-1]\cap\Lambda^{c}}(|S_{f(k)}|-|S_{k}|)+\sum_{j\in[1,k_{0}-1]\cap\Lambda\backslash\{f(k):k\in[1,k_{0}-1]\cap\Lambda^{c}\}}|S_{j}|\right)
\\\geq&\frac{1}{2}(0+|[1,k_{0}-1]\cap\Lambda\backslash\{f(k):k\in[1,k_{0}-1]\cap\Lambda^{c}\}||S_{1}|).
\end{align*}
The last equality follows from the fact that $f(k)\geq k$ for all $k\in[1,k_{0}-1]\cap\Lambda^{c}$ and the assumption that $|S_{1}|\leq\ldots\leq|S_{r}|$. Therefore,
\begin{eqnarray*}
\sum_{k\in\Lambda}|S_{k}\backslash(S_{1}\cup\ldots\cup S_{k-1})|&\geq&\frac{1}{2}(|[1,k_{0}-1]\cap\Lambda|-|\{f(k):k\in[1,k_{0}-1]\cap\Lambda^{c}\}|)|S_{1}|\\&=&
\frac{1}{2}(|[1,k_{0}-1]\cap\Lambda|-|[1,k_{0}-1]\cap\Lambda^{c}|)|S_{1}|\\&=&
\frac{1}{2}(|\Lambda|-|[k_{0},r]\cap\Lambda|-|\Lambda^{c}|+|[k_{0},r]\cap\Lambda^{c}|)|S_{1}|\\&\geq&
\frac{1}{2}(|\Lambda|-|\Lambda^{c}|)|S_{1}|.
\end{eqnarray*}
The last inequality follows from Case II assumption.
\end{proof}
\begin{proof}[Proof of Lemma \ref{21}]
Let $\Lambda=\{1\leq k\leq r:|S_{k}\backslash(S_{1}\cup\ldots\cup S_{k-1})|\leq t\}$. Then
\begin{align*}
&\min(t,|S_{1}|)+\min(t,|S_{2}\backslash S_{1}|)+\ldots+\min(t,|S_{n}\backslash(S_{1}\cup\ldots\cup S_{n-1})|)\\=&\sum_{k\in\Lambda}|S_{k}\backslash(S_{1}\cup\ldots\cup
S_{k-1})|+t|\Lambda^{c}|.
\end{align*}
If $|\Lambda^{c}|\geq\frac{r}{2}$ then
\[\min(t,|S_{1}|)+\min(t,|S_{2}\backslash S_{1}|)+\ldots+\min(t,|S_{r}\backslash(S_{1}\cup\ldots\cup S_{r-1})|)\geq\frac{tr}{2}\]
and the result follows. If $|\Lambda|\geq\frac{r}{2}$ then $|\Lambda|-|\Lambda^{c}|\geq 0$ so by Lemma \ref{25}, it follows that
\begin{align*}
&\min(t,|S_{1}|)+\min(t,|S_{2}\backslash S_{1}|)+\ldots+\min(t,|S_{n}\backslash(S_{1}\cup\ldots\cup S_{n-1})|)\\\geq&\frac{1}{2}|S_{1}|(|\Lambda|-|\Lambda^{c}|)+
t|\Lambda^{c}|\\\geq&\frac{1}{2}\min(t,|S_{1}|)(|\Lambda|-|\Lambda^{c}|)+\min(t,|S_{1}|)|\Lambda^{c}|=\frac{1}{2}\min(t,|S_{1}|)(|\Lambda|+|\Lambda^{c}|)=
\frac{\min(t,|S_{1}|)}{2}r.
\end{align*}
\end{proof}
\section{Proof of the main result}
\begin{lemma}\label{30}
If $y$ and $z$ are nonnegative random variables then for every $0<\epsilon<1$,
\[\mathbb{E}yz\leq(\mathbb{E}y)^{1-\epsilon}(\mathbb{E}y(z^{\frac{1}{\epsilon}}))^{\epsilon}.\]
\end{lemma}
\begin{proof}
By H\"older's inequality, $\mathbb{E}yz=\mathbb{E}y^{1-\epsilon}(y^{\epsilon}z)\leq(\mathbb{E}y)^{1-\epsilon}(\mathbb{E}(y^{\epsilon}z)^{\frac{1}{\epsilon}})^{\epsilon}=(\mathbb{E}y)^{1-\epsilon}(\mathbb{E}y(z^{\frac{1}{\epsilon}}))^{\epsilon}$.
\end{proof}
\begin{lemma}\label{31}
Let $f_{1},\ldots,f_{r}$ be a random vector on $\mathbb{C}^{n}$ such that for every $\delta>0$ there exists $M_{\delta}>0$ such that
\[\sup_{x\in S^{n-1}}\mathbb{E}|(f,x)|^{4}\leq\frac{M_{\delta}}{n^{2(1-\delta)}}\text{ and }\mathbb{E}\|f\|^{k}\leq L_{k},\quad f\in\{f_{1},\ldots,f_{r}\},\;k\geq 1.\]
Then for every $\epsilon>0$ and $x_{1},\ldots,x_{r}\in\mathbb{C}^{n}$ with $\|x_{i}\|\leq 1$,
\[\mathbb{E}|(f_{1},x_{1})|\ldots|(f_{r},x_{r})|\leq\frac{C_{\epsilon}}{n^{\frac{1}{2}\min(r,4)(1-\epsilon)}},\]
where $C_{\epsilon}$ depends on $\epsilon$ and certain $M_{\delta}$ and $L_{k,\delta}$ but not on $n$.
\end{lemma}
\begin{proof}
By H\"older's inequality,
\[\mathbb{E}|(f_{1},x_{1})|\ldots|(f_{r},x_{r})|\leq(\mathbb{E}|(f_{1},x_{1})|^{r})^{\frac{1}{r}}\ldots(\mathbb{E}|(f_{r},x_{r})|^{r})^{\frac{1}{r}}\]
so it suffices to prove the lemma when $f_{1}=\ldots=f_{r}=f$ and $x_{1}=\ldots=x_{r}=x$. If $r>4$ then by Lemma \ref{30}, for every $\epsilon>0$,
\begin{eqnarray*}
\mathbb{E}|(f,x)|^{r}&\leq&\mathbb{E}|(f,x)|^{4}\|f\|^{r-4}\\&\leq&(\mathbb{E}|(f,x)|^{4})^{1-\frac{\epsilon}{2}}(\mathbb{E}|(f,x)|^{4}\|f\|^{\frac{2(r-4)}{\epsilon}})^{\frac{\epsilon}{2}}\\&\leq&(\mathbb{E}|(f,x)|^{4})^{1-\frac{\epsilon}{2}}(\mathbb{E}\|f\|^{4+\frac{2(r-4)}{\epsilon}})^{\frac{\epsilon}{2}}\\&\leq&\left(\frac{M_{\frac{\epsilon}{2}}}{n^{2(1-\frac{\epsilon}{2})}}\right)^{1-\frac{\epsilon}{2}}(L_{4+\frac{2(r-4)}{\epsilon}})^{\frac{\epsilon}{2}}\\&\leq&
\frac{M_{\frac{\epsilon}{2}}^{^{1-\frac{\epsilon}{2}}}}{n^{2(1-\epsilon)}}(L_{4+\frac{2(r-4)}{\epsilon}})^{\frac{\epsilon}{2}}.
\end{eqnarray*}
If $r\leq 4$ then by H\"older's inequality,
\[\mathbb{E}|(f,x)|^{r}\leq(\mathbb{E}|(f,x)|^{4})^{\frac{r}{4}}\leq\frac{M_{\epsilon}^{\frac{r}{4}}}{n^{\frac{r}{2}(1-\epsilon)}}.\]
\end{proof}
\begin{lemma}\label{Estimate}
Let $G=(V,E)$ be a graph with no loops but perhaps with multiple edges. Let $(\mathscr{B}_{v})_{v\in V}$ be independent $\sigma$-subalgebras of a probability space $(\Omega,\mathscr{B},\mathbb{P})$. For each $e\in E$, let $u_{1}(e)$ and $u_{2}(e)$ be the two endpoints of $e$ and let $h_{e}^{(1)}$ and $h_{e}^{(2)}$ be $\mathscr{B}_{u_{1}(e)}$-measurable and $\mathscr{B}_{u_{2}(e)}$-measurable random vectors on $\mathbb{C}^{n}$. Assume that for every $\delta>0$, there exist $M_{\delta}>0$ and $L_{k,\delta}$ such that
\[\sup_{x\in S^{n-1}}\mathbb{E}|(h,x)|^{4}\leq\frac{M_{\delta}}{n^{2(1-\delta)}}\text{ and }\mathbb{E}\|h\|^{k}\leq L_{k,\delta}n^{\delta},\quad h\in\bigcup_{e\in E}\{h_{e}^{(1)},h_{e}^{(2)}\},\;k\geq 1.\]
If every vertex has degree at least $4$, then for every $\epsilon>0$,
\[\mathbb{E}\prod_{e\in E}|\langle h_{e}^{(1)},h_{e}^{(2)}\rangle|\leq\frac{C_{\epsilon}}{n^{|V|(1-\epsilon)}},\]
where $C_{\epsilon}$ depends on $\epsilon$, the graph $G$ and certain $M_{\delta}$ and $L_{k,\delta}$ but not on $n$.
\end{lemma}
\begin{proof}
Let $v_{1},\ldots,v_{|V|}$ be an enumeration of $V$ with ascending order according to their degrees, i.e., defining $S_{j}$ to be the set of all edges incident to $v_{j}$, we have $|S_{1}|\leq|S_{2}|\leq\ldots|S_{|V|}|$. For each $j=1,\ldots,|V|$, if $e\in S_{j}\backslash(S_{1}\cup\ldots\cup S_{j-1})$ then either $u_{1}(e)=v_{j}$ or $u_{2}(e)=v_{j}$ and so by interchanging the values of $u_{1}(e)$ and $u_{2}(e)$ (and accordingly also $h_{e}^{(1)}$ and $h_{e}^{(2)}$), if necessary, we may assume that $u_{1}(e)=v_{j}$. Thus, for every $\eta>0$
\begin{align}\label{First}
&\mathbb{E}\prod_{e\in E}|\langle h_{e}^{(1)},h_{e}^{(2)}\rangle|\nonumber\\=&\mathbb{E}\prod_{j=1}^{|V|}\prod_{e\in S_{j}\backslash(S_{1}\cup\ldots\cup S_{j-1})}|\langle h_{e}^{(1)},h_{e}^{(2)}\rangle|\nonumber\\=&
\mathbb{E}\prod_{j=1}^{|V|}\prod_{e\in S_{j}\backslash(S_{1}\cup\ldots\cup S_{j-1})}\left|\left\langle h_{e}^{(1)},\frac{h_{e}^{(2)}}{\|h_{e}^{(2)}\|+\eta}\right\rangle\right|(\|h_{e}^{(2)}\|+\eta)\nonumber\\=&
\mathbb{E}\prod_{j=1}^{|V|}\prod_{e\in S_{j}\backslash(S_{1}\cup\ldots\cup S_{j-1})}\left|\left\langle h_{e}^{(1)},\frac{h_{e}^{(2)}}{\|h_{e}^{(2)}\|+\eta}\right\rangle\right|\prod_{j=1}^{|V|}\prod_{e\in S_{j}\backslash(S_{1}\cup\ldots\cup S_{j-1})}(\|h_{e}^{(2)}\|+\eta)\nonumber\\=&
\mathbb{E}\prod_{j=1}^{|V|}\prod_{e\in S_{j}\backslash(S_{1}\cup\ldots\cup S_{j-1})}\left|\left\langle h_{e}^{(1)},\frac{h_{e}^{(2)}}{\|h_{e}^{(2)}\|+\eta}\right\rangle\right|\prod_{e\in E}(\|h_{e}^{(2)}\|+\eta),
\end{align}
where as before, when $j=1$, $S_{j}\backslash(S_{1}\cup\ldots\cup S_{j-1})$ is understood as $S_{1}$.

Since $u_{1}(e)=v_{j}$, $h_{e}^{(1)}$ is $\mathscr{B}_{v_{j}}$-measurable. On the other hand, by assumption, $h_{e}^{(2)}$ is $\mathscr{B}_{u_{2}(e)}$-measurable; and since $G$ has no loops, $u_{2}(e)\neq u_{1}(e)=v_{j}$. Therefore, by Lemma \ref{31},
\begin{equation}\label{CondExp}
\mathbb{E}_{\mathscr{B}_{v_{j}}}\prod_{e\in S_{j}\backslash(S_{1}\cup\ldots\cup S_{j-1})}\left|\left\langle h_{e}^{(1)},\frac{h_{e}^{(2)}}{\|h_{e}^{(2)}\|+\eta}\right\rangle\right|\leq\frac{C_{\epsilon}}{n^{\frac{1}{2}\min(|S_{j}\backslash(S_{1}\cup\ldots S_{j-1})|,4)(1-\epsilon)}}.
\end{equation}
Note that the right hand side is a constant. We claim that
\begin{equation}\label{Claim}
\mathbb{E}\prod_{j=1}^{|V|}\prod_{e\in S_{j}\backslash(S_{1}\cup\ldots\cup S_{j-1})}\left|\left\langle h_{e}^{(1)},\frac{h_{e}^{(2)}}{\|h_{e}^{(2)}\|+\eta}\right\rangle\right|\leq\frac{C_{\epsilon}}{n^{|V|(1-\epsilon)}},
\end{equation}
where $C_{\epsilon}$ denotes any positive number depending on $\epsilon$, the graph $G$ and certain $M_{\delta}$ and $L_{k,\delta}$ but not on $n$.

To prove the claim, we write
\begin{align*}
&\mathbb{E}\prod_{j=1}^{|V|}\prod_{e\in S_{j}\backslash(S_{1}\cup\ldots\cup S_{j-1})}\left|\left\langle h_{e}^{(1)},\frac{h_{e}^{(2)}}{\|h_{e}^{(2)}\|+\eta}\right\rangle\right|\\=&\mathbb{E}\left(\prod_{e\in S_{1}}\left|\left\langle h_{e}^{(1)},\frac{h_{e}^{(2)}}{\|h_{e}^{(2)}\|+\eta}\right\rangle\right|\right)\left(\prod_{j=2}^{|V|}\prod_{e\in S_{j}\backslash(S_{1}\cup\ldots S_{j-1})}\left|\left\langle h_{e}^{(1)},\frac{h_{e}^{(2)}}{\|h_{e}^{(2)}\|+\eta}\right\rangle\right|\right).
\end{align*}
All the edges $e$ in the first parenthesis are incident to $v_{1}$, whereas all the $e$ in the second parenthesis are not incident to $v_{1}$. Thus, the term in the second parenthesis is independent of $\mathscr{B}_{v_{1}}$ and so
\begin{align*}
&\mathbb{E}\prod_{j=1}^{|V|}\prod_{e\in S_{j}\backslash(S_{1}\cup\ldots\cup S_{j-1})}\left|\left\langle h_{e}^{(1)},\frac{h_{e}^{(2)}}{\|h_{e}^{(2)}\|+\eta}\right\rangle\right|\\=&
\mathbb{E}\left(\mathbb{E}_{\mathscr{B}_{v_{1}}}\prod_{e\in S_{1}}\left|\left\langle h_{e}^{(1)},\frac{h_{e}^{(2)}}{\|h_{e}^{(2)}\|+\eta}\right\rangle\right|\right)\left(\prod_{j=2}^{|V|}\prod_{e\in S_{j}\backslash(S_{1}\cup\ldots S_{j-1})}\left|\left\langle h_{e}^{(1)},\frac{h_{e}^{(2)}}{\|h_{e}^{(2)}\|+\eta}\right\rangle\right|\right)\\\leq&
\frac{C_{\epsilon}}{n^{\frac{1}{2}\min(|S_{1}|,4)(1-\epsilon)}}\mathbb{E}\left(\prod_{j=2}^{|V|}\prod_{e\in S_{j}\backslash(S_{1}\cup\ldots S_{j-1})}\left|\left\langle h_{e}^{(1)},\frac{h_{e}^{(2)}}{\|h_{e}^{(2)}\|+\eta}\right\rangle\right|\right),
\end{align*}
where the inequality follows from (\ref{CondExp}). Continuing this procedure, we obtain
\begin{align*}
&\mathbb{E}\prod_{j=1}^{|V|}\prod_{e\in S_{j}\backslash(S_{1}\cup\ldots\cup S_{j-1})}\left|\left\langle h_{e}^{(1)},\frac{h_{e}^{(2)}}{\|h_{e}^{(2)}\|+\eta}\right\rangle\right|\\\leq&\frac{C_{\epsilon}}{n^{\frac{1}{2}\min(|S_{1}|,4)(1-\epsilon)}}\frac{C_{\epsilon}}{n^{\frac{1}{2}\min(|S_{2}\backslash S_{1}|,4)(1-\epsilon)}}\ldots\frac{C_{\epsilon}}{n^{\frac{1}{2}\min(|S_{|V|}\backslash(S_{1}\cup\ldots\cup S_{|V|-1})|,4)(1-\epsilon)}}.
\end{align*}
By Lemma \ref{21}, it follows that
\[\mathbb{E}\prod_{j=1}^{|V|}\prod_{e\in S_{j}\backslash(S_{1}\cup\ldots\cup S_{j-1})}\left|\left\langle h_{e}^{(1)},\frac{h_{e}^{(2)}}{\|h_{e}^{(2)}\|+\eta}\right\rangle\right|\\\leq\frac{C_{\epsilon}}{n^{\frac{1}{4}\min(|S_{1}|,4)|V|(1-\epsilon)}},\]
possibly with different $C_{\epsilon}$. Since by assumption, $|S_{1}|\geq 4$, the claim (\ref{Claim}) is proved.
Having proved (\ref{Claim}), before we apply Lemma \ref{30}, we estimate
\begin{align*}
&\mathbb{E}\left(\prod_{j=1}^{|V|}\prod_{e\in S_{j}\backslash(S_{1}\cup\ldots\cup S_{j-1})}\left|\left\langle h_{e}^{(1)},\frac{h_{e}^{(2)}}{\|h_{e}^{(2)}\|+\eta}\right\rangle\right|\right)\left(\prod_{e\in E}(\|h_{e}^{(2)}\|+\eta)\right)^{\frac{1}{\epsilon}}\\\leq&
\mathbb{E}\left(\prod_{j=1}^{|V|}\prod_{e\in S_{j}\backslash(S_{1}\cup\ldots\cup S_{j-1})}\|h_{e}^{(1)}\|\right)\left(\prod_{e\in E}(\|h_{e}^{(2)}\|+\eta)^{\frac{1}{\epsilon}}\right)\\\leq&
\mathbb{E}\left(\prod_{e\in E}\|h_{e}^{(1)}\|\right)\left(\prod_{e\in E}(\|h_{e}^{(2)}\|+\eta)^{\frac{1}{\epsilon}}\right)\\\leq&
\left(\prod_{e\in E}\mathbb{E}\|h_{e}^{(1)}\|^{2|E|}\prod_{e\in E}\mathbb{E}(\|h_{e}^{(2)}\|+\eta)^{\frac{2|E|}{\epsilon}}\right)^{\frac{1}{2|E|}},
\end{align*}
where the last inequality follows from H\"older's inequality. Combining this estimate with (\ref{First}), (\ref{Claim}) and Lemma \ref{30}, we obtain
\[\mathbb{E}\prod_{e\in E}|\langle h_{e}^{(1)},h_{e}^{(2)}\rangle|\leq\frac{C_{\epsilon}}{n^{|V|(1-\epsilon)^{2}}}\left(\prod_{e\in E}\mathbb{E}\|h_{e}^{(1)}\|^{2|E|}\prod_{e\in E}\mathbb{E}(\|h_{e}^{(2)}\|+\eta)^{\frac{2|E|}{\epsilon}}\right)^{\frac{\epsilon}{2|E|}}.\]
Taking $\eta$ to be arbitarily small, we have
\begin{eqnarray*}
\mathbb{E}\prod_{e\in E}|\langle h_{e}^{(1)},h_{e}^{(2)}\rangle|&\leq&\frac{C_{\epsilon}}{n^{|V|(1-\epsilon)^{2}}}\left(\prod_{e\in E}\mathbb{E}\|h_{e}^{(1)}\|^{2|E|}\prod_{e\in E}\mathbb{E}\|h_{e}^{(2)}\|^{\frac{2|E|}{\epsilon}}\right)^{\frac{\epsilon}{2|E|}}\\&\leq&
\frac{C_{\epsilon}}{n^{|V|(1-\epsilon)^{2}}}\left(\prod_{e\in E}(L_{2|E|,1}n)\prod_{e\in E}(L_{\frac{2|E|}{\epsilon},1}n)\right)^{\frac{\epsilon}{2|E|}}\\&\leq&
\frac{C_{\epsilon}}{n^{|V|(1-\epsilon)^{2}}}\left(\prod_{e\in E}L_{2|E|,1}\prod_{e\in E}L_{\frac{2|E|}{\epsilon},1}\right)^{\frac{\epsilon}{2|E|}}n^{\epsilon},
\end{eqnarray*}
where the second inequality follows from the assumption. This completes the proof with a different $\epsilon$.
\end{proof}
\begin{lemma}\label{Trace}
Suppose that $(\mathscr{B}_{j})_{j\in J}$ are independent $\sigma$-subalgebras of a probability space $(\Omega,\mathscr{B},\mathbb{P})$. Let $j:\{1,\ldots,p\}\to J$ be such that $\ker j$ is a crossing partition on $\{1,\ldots,p\}$. For each $i=1,\ldots,p$, let $f_{i}^{(1)},f_{i}^{(2)}$ be $\mathscr{B}_{j(i)}$-measurable functions on $\Omega$. Assume that for every $\delta>0$, there exist $M_{\delta}>0$ and $L_{k,\delta}>0$, $k\geq 1$ such that
\begin{equation}\label{31e}
\sup_{x\in S^{n-1}}\mathbb{E}|(f,x)|^{4}\leq\frac{M_{\delta}}{n^{2(1-\delta)}}\text{ and }\mathbb{E}\|f\|^{k}\leq L_{k,\delta}n^{\delta},\quad f\in\{f_{1}^{(1)},f_{1}^{(2)},\ldots,f_{p}^{(1)},f_{p}^{(2)}\},\;k\geq 1
\end{equation}
Then for every $\epsilon>0$,
\[|\mathbb{E}\circ\mathrm{tr}(f_{1}^{(1)}\otimes f_{1}^{(2)})(f_{2}^{(1)}\otimes f_{2}^{(2)})\ldots(f_{p}^{(1)}\otimes f_{p}^{(2)})|\leq\frac{C_{\epsilon}}{n^{|\{j(1),\ldots,j(p)\}|+1-\epsilon}},\]
where $C_{\epsilon}>0$ depends on $\epsilon,p$ and certain $M_{\delta}$ and $L_{k,\delta}$ but not on $n$.
\end{lemma}
\begin{proof}
We may assume that $j(1)\neq j(2)\neq\ldots\neq j(p)\neq j(1)$ and each $j(i)$ appears at least twice in the list $j(1),\ldots,j(p)$. Otherwise, if $j(i)=j(i+1)$ then
\[(f_{i}^{(1)}\otimes f_{i}^{(2)})(f_{i+1}^{(1)}\otimes f_{i+1}^{(2)})=\langle f_{i+1}^{(1)},f_{i}^{(2)}\rangle(f_{i}^{(1)}\otimes f_{i+1}^{(2)})=(\langle f_{i+1}^{(1)},f_{i}^{(2)}\rangle f_{i}^{(1)})\otimes f_{i+1}^{(2)}.\]
Note that $\langle f_{i+1}^{(1)},f_{i}^{(2)}\rangle f_{i}^{(1)}$ and $f_{i+1}^{(2)}$ are $\mathscr{B}_{j(i)}$-measurable since $j(i)=j(i+1)$. Also, by H\"older's inequality, $\langle f_{i+1}^{(1)},f_{i}^{(2)}\rangle f_{i}^{(1)}$ satisfies (\ref{31e}) perhaps with different $M_{\delta}$ and $L_{k,\delta}$. Thus, the result follows by induction hypothesis since the product $(f_{1}^{(1)}\otimes f_{1}^{(2)})\ldots(f_{p-1}^{(1)}\otimes f_{p}^{(2)})$ of $p$ terms becomes a product of $p-1$ terms. (The $i$th term and the $(i+1)$th term are combined.)

Similar argument works if we have $j(p)\neq j(1)$.

If there is a $j(i)$ that appears only once in the list $j(1),\ldots,j(p)$, then by independence of $(\mathscr{B}_{j})_{j\in J}$,
\begin{align}\label{32e}
&\mathbb{E}\circ\mathrm{tr}(f_{1}^{(1)}\otimes f_{1}^{(2)})\ldots(f_{p-1}^{(1)}\otimes f_{p}^{(2)})\nonumber\\=&
\mathbb{E}\circ\mathrm{tr}(f_{1}^{(1)}\otimes f_{1}^{(2)})\ldots(f_{i}^{(1)}\otimes f_{i}^{(2)})\ldots(f_{p-1}^{(1)}\otimes f_{p}^{(2)})\nonumber\\=&
\mathbb{E}\circ\mathrm{tr}(f_{1}^{(1)}\otimes f_{1}^{(2)})\ldots(\mathbb{E}f_{i}^{(1)}\otimes f_{i}^{(2)})\ldots(f_{p-1}^{(1)}\otimes f_{p}^{(2)})\nonumber\\=&
\mathbb{E}\circ\mathrm{tr}(f_{1}^{(1)}\otimes f_{1}^{(2)})\ldots(\mathbb{E}f_{i}^{(1)}\otimes f_{i}^{(2)})(f_{i+1}^{(1)}\otimes f_{i+1}^{(2)})\ldots(f_{p-1}^{(1)}\otimes f_{p}^{(2)})\nonumber\\=&
\mathbb{E}\circ\mathrm{tr}(f_{1}^{(1)}\otimes f_{1}^{(2)})\ldots((\mathbb{E}f_{i}^{(1)}\otimes f_{i}^{(2)})f_{i+1}^{(1)}\otimes f_{i+1}^{(2)})\ldots(f_{p}^{(1)}\otimes f_{p}^{(2)})\nonumber\\=&
\frac{1}{n}\mathbb{E}\circ\mathrm{tr}(f_{1}^{(1)}\otimes f_{1}^{(2)})\ldots(n(\mathbb{E}f_{i}^{(1)}\otimes f_{i}^{(2)})f_{i+1}^{(1)}\otimes f_{i+1}^{(2)})\ldots(f_{p}^{(1)}\otimes f_{p}^{(2)}).
\end{align}
Note that $\mathbb{E}f_{i}^{(1)}\otimes f_{i}^{(2)}$ is a deterministic matrix and
\begin{eqnarray*}
|\langle(\mathbb{E}f_{i}^{(1)}\otimes f_{i}^{(2)})x,y\rangle|&=&|\mathbb{E}\langle x,f_{i}^{(2)}\rangle\langle f_{i}^{(1)},y\rangle|\\&\leq&\mathbb{E}|\langle x,f_{i}^{(2)}\rangle||\langle f_{i}^{(1)},y\rangle|\\&\leq&(\mathbb{E}|\langle f_{i}^{(2)},x\rangle|^{2})^{\frac{1}{2}}(\mathbb{E}|\langle f_{i}^{(1)},y\rangle|^{2})^{\frac{1}{2}}\\&\leq&(\mathbb{E}|\langle f_{i}^{(2)},x\rangle|^{4})^{\frac{1}{4}}(\mathbb{E}|\langle f_{i}^{(1)},y\rangle|^{4})^{\frac{1}{4}}\\&\leq&\left(\frac{M_{\delta}}{n^{2(1-\delta)}}\right)^{\frac{1}{4}}\left(\frac{M_{\delta}}{n^{2(1-\delta)}}\right)^{\frac{1}{4}}\\&=&\frac{\sqrt{M_{\delta}}}{n^{1-\delta}},\quad x,y\in S^{n-1}.
\end{eqnarray*}
Thus,
\[\|n\mathbb{E}f_{i}^{(1)}\otimes f_{i}^{(2)}\|\leq\sqrt{M_{\delta}}n^{\delta}.\]
Hence, $n(\mathbb{E}f_{i}^{(1)}\otimes f_{i}^{(2)})f_{i+1}^{(1)}$ is $\mathscr{B}_{j(i+1)}$ and still satisfies (\ref{31e}) perhaps with different $M_{\delta}$ and $L_{k,\delta}$. Thus, in view of (\ref{32e}), the result follows by induction hypothesis since the product $(f_{1}^{(1)}\otimes f_{1}^{(2)})\ldots(f_{p}^{(1)}\otimes f_{p}^{(2)})$ of $p$ terms becomes a product of $p-1$ terms. (The $i$th term is absorbed by the $(i+1)$th term.)

Therefore, we may justifiably assume that $j(1)\neq j(2)\neq\ldots\neq j(p)\neq j(1)$ and each $j(i)$ appears at least twice in the list $j(1),\ldots,j(p)$.
\begin{align*}
&|\mathbb{E}\circ\mathrm{tr}(f_{1}^{(1)}\otimes f_{1}^{(2)})(f_{2}^{(1)}\otimes f_{2}^{(2)})\ldots(f_{p}^{(1)}\otimes f_{p}^{(2)})|\\=&\frac{1}{n}|\mathbb{E}\langle f_{1}^{(2)},f_{2}^{(1)}\rangle\langle f_{2}^{(2)},f_{3}^{(1)}\rangle\ldots\langle f_{p}^{(2)},f_{1}^{(1)}\rangle|\leq\frac{1}{n}\mathbb{E}|\langle f_{1}^{(2)},f_{2}^{(1)}\rangle||\langle f_{2}^{(2)},f_{3}^{(1)}\rangle|\ldots|\langle f_{p}^{(2)},f_{1}^{(1)}\rangle|.
\end{align*}
For notational convenience, let $j(p+1)=j(1)$ and $f_{p+1}^{(1)}=f_{1}^{(1)}$. Then we have
\begin{equation}\label{FirstEstimate}
|\mathbb{E}\circ\mathrm{tr}(f_{1}^{(1)}\otimes f_{1}^{(2)})(f_{2}^{(1)}\otimes f_{2}^{(2)})\ldots(f_{p}^{(1)}\otimes f_{p}^{(2)})|\leq\frac{1}{n}\mathbb{E}\prod_{i=1}^{p}|\langle f_{i}^{(2)},f_{i+1}^{(1)}\rangle|.
\end{equation}
We use Lemma \ref{Estimate} to estimate this. First, we take the vertex set $V=\{j(1),\ldots,j(p)\}$ and the edge set $E=\{1,\ldots,p\}$, where for each $i\in E$, the two endpoints are $u_{1}(i)=j(i)$ and $u_{2}(i)=j(i+1)$. There are no loops since we assume that $j(i)\neq j(i+1)$ for all $i=1,\ldots,p$. For each $i\in E$, take $h_{i}^{(1)}=f_{i}^{(2)}$ and $h_{i}^{(2)}=f_{i+1}^{(1)}$. To see that every vertex has degree at least $4$, recall that we assume that for every $j\in V=\{j(1),\ldots,j(p)\}$, there exist $i_{1}\neq i_{2}$ in $\{1,\ldots,p\}$ such that $j(i_{1})=j(i_{2})=j$. Since $j(1)\neq j(2)\neq\ldots\neq j(p)\neq j(1)$, $i_{1}$ and $i_{2}$ cannot be consective numbers. Therefore, the vertex $j$ is incident with the four distinct edges $i_{1}-1,i_{1},i_{2}-1,i_{2}$. (When $i_{1}=1$, $i_{1}-1=p$.) Thus, the assumptions of Lemma \ref{Estimate} are satisfied and so we obtain
\[\mathbb{E}\prod_{i=1}^{p}|\langle f_{i}^{(2)},f_{i+1}^{(1)}\rangle|\leq\frac{C_{\epsilon}}{n^{|\{j(1),\ldots,j(p)\}|(1-\epsilon)}}.\]
The result follows by combining this with \ref{FirstEstimate}.
\end{proof}
\begin{remark}
In Lemma \ref{Trace}, the assumption that $\ker j$ is a crossing partition is necessary because it guarantees that repeating the procedure of (1) combining the $i$th term and the $(i+1)$th term when $j(i)=j(i+1)$ and (2) the $i$th term being absorbed by the $(i+1)$th term when $j(i)$ appears only once in the list $j(1),\ldots,j(p)$ does not make reduce $\{1,\ldots,p\}$ to a singleton. Without the crossing assumption, one would have got Lemma \ref{BoundExp} below.
\end{remark}
As an immediate consequence of Lemma \ref{Trace}, we have
\begin{proposition}\label{FirstProp}
Suppose that $(f_{j})_{j\in J}$ is an independent family of random vectors on $\mathbb{C}^{n}$ such that
\[\sup_{x\in S^{n-1}}\mathbb{E}|(f_{j},x)|^{4}\leq\frac{L}{n^{2}}\text{ and }\mathbb{E}\|f_{j}\|^{k}\leq L_{k},\quad j\in J,\;k\geq 1\]
for some $L>0$ and $L_{k}>0$, $k\geq 1$ independent of $N$. Let $j:\{1,\ldots,p\}\to J$ be such that $\ker j$ is a crossing partition on $\{1,\ldots,p\}$. Then for every $\epsilon>0$,
\[|\mathbb{E}\circ\mathrm{tr}(f_{j(1)}\otimes f_{j(1)})\ldots(f_{j(p)}\otimes f_{j(p)})|\leq\frac{C_{\epsilon}}{n^{|\{j(1),\ldots,j(p)\}|+1-\epsilon}},\]
where $C_{\epsilon}>0$ depends on $\epsilon,p,L$ and certain $L_{k}$ but not on $n$.
\end{proposition}
The following lemma is the analog of Lemma \ref{Trace} for noncrossing partition. 
\begin{lemma}\label{BoundExp}
Suppose that $(\mathscr{B}_{j})_{j\in J}$ are independent $\sigma$-subalgebras of a probability space $(\Omega,\mathscr{B},\mathbb{P})$. Let $j:\{1,\ldots,p\}\to J$ be such that $\ker j$ is a noncrossing partition on $\{1,\ldots,p\}$. For each $i=1,\ldots,p$, let $f_{i}^{(1)},f_{i}^{(2)}$ be $\mathscr{B}_{j(i)}$-measurable functions on $\Omega$. Assume that for every $\delta>0$, there exist $M_{\delta}>0$ and $L_{k,\delta}>0$, $k\geq 1$ such that
\[\sup_{x\in S^{n-1}}\mathbb{E}|(f,x)|^{4}\leq\frac{M_{\delta}}{n^{2(1-\delta)}}\text{ and }\mathbb{E}\|f\|^{k}\leq L_{k,\delta}n^{\delta},\quad f\in\{f_{1}^{(1)},f_{1}^{(2)},\ldots,f_{p}^{(1)},f_{p}^{(2)}\},\;k\geq 1.\]
Then for every $\epsilon>0$,
\begin{equation}\label{BoundExpe}
\|\mathbb{E}(f_{1}^{(1)}\otimes f_{1}^{(2)})(f_{2}^{(1)}\otimes f_{2}^{(2)})\ldots(f_{p}^{(1)}\otimes f_{p}^{(2)})\|\leq\frac{C_{\epsilon}}{n^{|\{j(1),\ldots,j(p)\}|-\epsilon}},
\end{equation}
where $C_{\epsilon}>0$ depends on $\epsilon,p$ and certain $M_{\delta}$ and $L_{k,\delta}$ but not on $n$.
\end{lemma}
The only differences are that on the left hand side of (\ref{BoundExpe}), one has norm of expectation instead of trace expectation and that on the right hand side of (\ref{BoundExpe}), one only has $\frac{C_{\epsilon}}{n^{|\{j(1),\ldots,j(p)\}|-\epsilon}}$ instead of $\frac{C_{\epsilon}}{n^{|\{j(1),\ldots,j(p)\}|+1-\epsilon}}$ in Lemma \ref{Trace}. The proof of Lemma \ref{BoundExp} is exactly the same as the beginning of the proof of Lemma \ref{Trace}. One needs the fact that for every noncrossing partition $\pi$ on $\{1,\ldots,p\}$, at least one of the following holds.
\begin{enumerate}
\item There exists $i\in\{1,\ldots,p-1\}$ such that $i$ and $i+1$ are in the same block of $\pi$.
\item $\pi$ has a singleton block.
\end{enumerate}
This is because every noncrossing partition contains an interval block.

As an immediate consequence of Lemma \ref{BoundExpe}, we have
\begin{lemma}\label{BoundExp2}
Suppose that $(f_{j})_{j\in J}$ is an independent family of random vectors on $\mathbb{C}^{n}$ such that
\[\sup_{x\in S^{n-1}}\mathbb{E}|(f_{j},x)|^{4}\leq\frac{L}{n^{2}}\text{ and }\mathbb{E}\|f_{j}\|^{k}\leq L_{k},\quad j\in J,\;k\geq 1\]
for some $L>0$ and $L_{k}>0$, $k\geq 1$ independent of $N$. Let $j:\{1,\ldots,p\}\to J$ be such that $\ker j$ is a noncrossing partition on $\{1,\ldots,p\}$. Then for every $\epsilon>0$,
\[\|\mathbb{E}(f_{j(1)}\otimes f_{j(1)})\ldots(f_{j(p)}\otimes f_{j(p)})\|\leq\frac{C_{\epsilon}}{n^{|\{j(1),\ldots,j(p)\}|-\epsilon}},\]
where $C_{\epsilon}>0$ depends on $\epsilon,p,L$ and certain $L_{k}$ but not on $n$.
\end{lemma}
\begin{proposition}\label{SecondProp}
Suppose that $f_{1},\ldots,f_{N}$ are independent random vectors on $\mathbb{C}^{n}$ such that
\[\sup_{x\in S^{n-1}}\mathbb{E}|(f_{j},x)|^{4}\leq\frac{L}{n^{2}}\text{ and }\mathbb{E}\|f_{j}\|^{k}\leq L_{k},\quad j=1,\ldots,N,\;k\geq 1\]
for some $L>0$ and $L_{k}>0$, $k\geq 1$ independent of $n$ and $N$. If $n,N\to\infty$ in such a way that $\frac{n}{N}\to\lambda\in(0,\infty)$ and\[n^{\epsilon_{0}}\left\|\sum_{j=1}^{N}\mathbb{E}\|f_{j}\|^{2(k-1)}f_{j}\otimes f_{j}-a_{k}I\right\|\to 0,\quad k\geq 1,\]
for some $a_{k}\in\mathbb{C}$, $k\geq 1$ and $\epsilon_{0}>0$ independent of $n$ and $N$, then for every noncrossing partition $\pi$ on $\{1,\ldots,p\}$,
\[\left|\sum_{\substack{j:\{1,\ldots,p\}\to\{1,\ldots,N\}\\\ker j=\pi}}\mathbb{E}\circ\mathrm{tr}(f_{j(1)}\otimes f_{j(1)})\ldots(f_{j(p)}\otimes f_{j(p)})-\prod_{B\in\pi}a_{|B|}\right|\to 0.\]
\end{proposition}
\begin{proof}
We prove by induction on $p$. For $p=1$, the result is obvious. For $p\geq 2$, since $\pi$ is a noncrossing partition on $\{1,\ldots,p\}$, there is an interval block $B_{0}\in\pi$. For simplicity, since the trace is cyclic invariant, we may assume that $B_{0}=\{1,\ldots,q\}$ for some $1\leq q\leq p$. Thus, for every $j:\{1,\ldots,p\}\to\{1,\ldots,N\}$ with $\ker j=\pi$, we have
\begin{align*}
&\mathbb{E}\circ\mathrm{tr}(f_{j(1)}\otimes f_{j(1)})\ldots(f_{j(p)}\otimes f_{j(p)})\\=&
\mathrm{tr}\mathbb{E}(f_{j(1)}\otimes f_{j(1)})\ldots(f_{j(q)}\otimes f_{j(q)})\mathbb{E}(f_{j(q+1)}\otimes f_{j(q+1)})\ldots(f_{j(p)}\otimes f_{j(p)})\\=&
\mathrm{tr}\mathbb{E}(\|f_{j(1)}\|^{2(q-1)}f_{j(1)}\otimes f_{j(1)})\mathbb{E}(f_{j(q+1)}\otimes f_{j(q+1)})\ldots(f_{j(p)}\otimes f_{j(p)}),
\end{align*}
since $j(1)=\ldots=j(q)$. Note that every $j:\{1,\ldots,p\}\to\{1,\ldots,N\}$ with $\ker j=\pi$ corresponds to $j:\{q+1,\ldots,p\}\to\{1,\ldots,N\}$ with $\ker l=\pi\backslash\{B_{0}\}$ and $j(1)\in\{1,\ldots,N\}\backslash\{j(q+1),\ldots,j(p)\}$. Thus,
\begin{align*}
&\sum_{\substack{j:\{1,\ldots,p\}\to\{1,\ldots,N\}\\\ker j=\pi}}\mathbb{E}\circ\mathrm{tr}(f_{j(1)}\otimes f_{j(1)})\ldots(f_{j(p)}\otimes f_{j(p)})\\=&
\sum_{\substack{j:\{q+1,\ldots,p\}\to\{1,\ldots,N\}\\\ker j=\pi\backslash\{B\}}}\sum_{j(1)\in\{1,\ldots,N\}\backslash\{j(q+1),\ldots,j(p)\}}\\&
\mathrm{tr}\mathbb{E}(\|f_{j(1)}\|^{2(q-1)}f_{j(1)}\otimes f_{j(1)})\mathbb{E}(f_{j(q+1)}\otimes f_{j(q+1)})\ldots(f_{j(p)}\otimes f_{j(p)})\\=&
\sum_{\substack{j:\{q+1,\ldots,p\}\to\{1,\ldots,N\}\\\ker j=\pi\backslash\{B\}}}\sum_{j(1)\in\{1,\ldots,N\}}\\&
\mathrm{tr}\mathbb{E}(\|f_{j(1)}\|^{2(q-1)}f_{j(1)}\otimes f_{j(1)})\mathbb{E}(f_{j(q+1)}\otimes f_{j(q+1)})\ldots(f_{j(p)}\otimes f_{j(p)})\\
&-\sum_{\substack{j:\{q+1,\ldots,p\}\to\{1,\ldots,N\}\\\ker j=\pi\backslash\{B\}}}\sum_{j(1)\in\{j(q+1),\ldots,j(p)\}}\\&
\mathrm{tr}\mathbb{E}(\|f_{j(1)}\|^{2(q-1)}f_{j(1)}\otimes f_{j(1)})\mathbb{E}(f_{j(q+1)}\otimes f_{j(q+1)})\ldots(f_{j(p)}\otimes f_{j(p)})\\
=&
\mathrm{tr}\left(\sum_{j(1)\in\{1,\ldots,N\}}\mathbb{E}\|f_{j(1)}\|^{2(q-1)}f_{j(1)}\otimes f_{j(1)}\right)\\&\left(\sum_{\substack{j:\{q+1,\ldots,p\}\to\{1,\ldots,N\}\\\ker j=\pi\backslash\{B\}}}\mathbb{E}(f_{j(q+1)}\otimes f_{j(q+1)})\ldots(f_{j(p)}\otimes f_{j(p)})\right)\\
&-\sum_{\substack{j:\{q+1,\ldots,p\}\to\{1,\ldots,N\}\\\ker j=\pi\backslash\{B\}}}\sum_{j(1)\in\{j(q+1),\ldots,j(p)\}}\\&
\mathrm{tr}\mathbb{E}(\|f_{j(1)}\|^{2(q-1)}f_{j(1)}\otimes f_{j(1)})\mathbb{E}(f_{j(q+1)}\otimes f_{j(q+1)})\ldots(f_{j(p)}\otimes f_{j(p)})\\=&
\mathrm{tr}a_{q}I\left(\sum_{\substack{j:\{q+1,\ldots,p\}\to\{1,\ldots,N\}\\\ker j=\pi\backslash\{B\}}}\mathbb{E}(f_{j(q+1)}\otimes f_{j(q+1)})\ldots(f_{j(p)}\otimes f_{j(p)})\right)\\
&+\mathrm{tr}\left(\sum_{j(1)\in\{1,\ldots,N\}}\mathbb{E}\|f_{j(1)}\|^{2(q-1)}f_{j(1)}\otimes f_{j(1)}-a_{q}I\right)\\&\left(\sum_{\substack{j:\{q+1,\ldots,p\}\to\{1,\ldots,N\}\\\ker j=\pi\backslash\{B\}}}\mathbb{E}(f_{j(q+1)}\otimes f_{j(q+1)})\ldots(f_{j(p)}\otimes f_{j(p)})\right)\\
&-\sum_{\substack{j:\{q+1,\ldots,p\}\to\{1,\ldots,N\}\\\ker j=\pi\backslash\{B\}}}\sum_{j(1)\in\{j(q+1),\ldots,j(p)\}}\\&
\mathrm{tr}\mathbb{E}(\|f_{j(1)}\|^{2(q-1)}f_{j(1)}\otimes f_{j(1)})\mathbb{E}(f_{j(q+1)}\otimes f_{j(q+1)})\ldots(f_{j(p)}\otimes f_{j(p)}).
\end{align*}
By induction hypothesis, the first term
\[\mathrm{tr}a_{q}I\left(\sum_{\substack{j:\{q+1,\ldots,p\}\to\{1,\ldots,N\}\\\ker j=\pi\backslash\{B\}}}\mathbb{E}(f_{j(q+1)}\otimes f_{j(q+1)})\ldots(f_{j(p)}\otimes f_{j(p)})\right)\]
converges to $\displaystyle a_{q}\prod_{B\in\pi\backslash\{B_{0}\}}a_{|B|}=\prod_{B\in\pi}a_{|B|}$. For the second term,
\begin{align*}
&\bigg|\mathrm{tr}\left(\sum_{j(1)\in\{1,\ldots,N\}}\mathbb{E}\|f_{j(1)}\|^{2(q-1)}f_{j(1)}\otimes f_{j(1)}-a_{q}I\right)\\&\left(\sum_{\substack{j:\{q+1,\ldots,p\}\to\{1,\ldots,N\}\\\ker j=\pi\backslash\{B\}}}\mathbb{E}(f_{j(q+1)}\otimes f_{j(q+1)})\ldots(f_{j(p)}\otimes f_{j(p)})\right)\bigg|\\\leq&
\left\|\sum_{j(1)\in\{1,\ldots,N\}}\mathbb{E}\|f_{j(1)}\|^{2(q-1)}f_{j(1)}\otimes f_{j(1)}-a_{q}I\right\|\\&
\left(\sum_{\substack{j:\{q+1,\ldots,p\}\to\{1,\ldots,N\}\\\ker j=\pi\backslash\{B\}}}\|\mathbb{E}(f_{j(q+1)}\otimes f_{j(q+1)})\ldots(f_{j(p)}\otimes f_{j(p)})\|\right)\\\leq&
\left\|\sum_{j(1)\in\{1,\ldots,N\}}\mathbb{E}\|f_{j(1)}\|^{2(q-1)}f_{j(1)}\otimes f_{j(1)}-a_{q}I\right\|\\&
\left(\sum_{\substack{j:\{q+1,\ldots,p\}\to\{1,\ldots,N\}\\\ker j=\pi\backslash\{B\}}}\frac{C_{\frac{\epsilon_{0}}{2}}}{n^{|\{j(q+1),\ldots,j(p)\}|-\frac{\epsilon_{0}}{2}}}\right)\\
&\text{ by Lemma }\ref{BoundExp2}\\\leq&
\left\|\sum_{j(1)\in\{1,\ldots,N\}}\mathbb{E}\|f_{j(1)}\|^{2(q-1)}f_{j(1)}\otimes f_{j(1)}-a_{q}I\right\|C_{\frac{\epsilon_{0}}{2}}n^{\frac{\epsilon_{0}}{2}}\to 0.
\end{align*}
For the third term,
\begin{align*}
&\bigg|\sum_{\substack{j:\{q+1,\ldots,p\}\to\{1,\ldots,N\}\\\ker j=\pi\backslash\{B\}}}\sum_{j(1)\in\{j(q+1),\ldots,j(p)\}}\\&
\mathrm{tr}\mathbb{E}(\|f_{j(1)}\|^{2(q-1)}f_{j(1)}\otimes f_{j(1)})\mathbb{E}(f_{j(q+1)}\otimes f_{j(q+1)})\ldots(f_{j(p)}\otimes f_{j(p)})\bigg|\\\leq&
\sum_{\substack{j:\{q+1,\ldots,p\}\to\{1,\ldots,N\}\\\ker j=\pi\backslash\{B\}}}\sum_{j(1)\in\{j(q+1),\ldots,j(p)\}}\\&
\|\mathbb{E}(\|f_{j(1)}\|^{2(q-1)}f_{j(1)}\otimes f_{j(1)})\|\|\mathbb{E}(f_{j(q+1)}\otimes f_{j(q+1)})\ldots(f_{j(p)}\otimes f_{j(p)})\|\\&
\sum_{\substack{j:\{q+1,\ldots,p\}\to\{1,\ldots,N\}\\\ker j=\pi\backslash\{B\}}}\sum_{j(1)\in\{j(q+1),\ldots,j(p)\}}
\frac{C_{\frac{1}{4}}}{n^{1-\frac{1}{4}}}\frac{C_{\frac{1}{4}}}{n^{|\{j(q+1),\ldots,j(p)\}|-\frac{1}{4}}}\\
&\text{ by Lemma }\ref{BoundExp2}\text{ with }\epsilon=\frac{1}{4}\\\leq&
\sum_{\substack{j:\{q+1,\ldots,p\}\to\{1,\ldots,N\}\\\ker j=\pi\backslash\{B\}}}
p\frac{C_{\frac{1}{4}}}{n^{1-\frac{1}{4}}}\frac{C_{\frac{1}{4}}}{n^{|\{j(q+1),\ldots,j(p)\}|-\frac{1}{4}}}\\=&
\sum_{\substack{j:\{q+1,\ldots,p\}\to\{1,\ldots,N\}\\\ker j=\pi\backslash\{B\}}}
\frac{C}{n^{|\{j(q+1),\ldots,j(p)\}|+\frac{1}{2}}}\leq\frac{C}{n^{\frac{1}{2}}}\to 0.
\end{align*}
\end{proof}
\begin{proof}[Proof of Theorem \ref{11}]
\begin{eqnarray*}
\mathbb{E}\circ\mathrm{tr}(f_{1}\otimes f_{1}+\ldots+f_{N}\otimes f_{N})^{p}&=&\sum_{j:\{1,\ldots,p\}\to\{1,\ldots,N\}}\mathbb{E}\circ\mathrm{tr}(f_{j(1)}\otimes f_{j(1)})\ldots(f_{j(p)}\otimes f_{j(p)})\\&=&\sum_{\substack{j:\{1,\ldots,p\}\to\{1,\ldots,N\}\\\ker j\text{ noncrossing}}}\mathbb{E}\circ\mathrm{tr}(f_{j(1)}\otimes f_{j(1)})\ldots(f_{j(p)}\otimes f_{j(p)})\\&&+
\sum_{\substack{j:\{1,\ldots,p\}\to\{1,\ldots,N\}\\\ker j\text{ crossing}}}\mathbb{E}\circ\mathrm{tr}(f_{j(1)}\otimes f_{j(1)})\ldots(f_{j(p)}\otimes f_{j(p)})
\end{eqnarray*}
The first term converges to $\displaystyle\sum_{\pi\in\mathrm{NC}(p)}\prod_{B\in\pi}a_{|B|}$ by Proposition \ref{SecondProp}. For the second term,
\begin{align*}
&\sum_{\substack{j:\{1,\ldots,p\}\to\{1,\ldots,N\}\\\ker j\text{ crossing}}}\mathbb{E}\circ\mathrm{tr}(f_{j(1)}\otimes f_{j(1)})\ldots(f_{j(p)}\otimes f_{j(p)})\\\leq&\sum_{\substack{j:\{1,\ldots,p\}\to\{1,\ldots,N\}\\\ker j\text{ crossing}}}\frac{C_{\frac{1}{2}}}{n^{|\{j(1),\ldots,j(p)\}|+1-\frac{1}{2}}}\text{ by Proposition }\ref{FirstProp}\text{ with }\epsilon=\frac{1}{2}\\\leq&\frac{C}{n^{\frac{1}{2}}}\to 0.
\end{align*}
\end{proof}

\end{document}